\documentclass[11pt, amssymb]{amsart}
\usepackage{amsmath, amsthm, amsfonts, verbatim, amssymb}
\usepackage{eucal}

\DeclareFontFamily{OT1}{rsfs}{}
\usepackage[all]{xy}
\CompileMatrices

\DeclareFontShape{OT1}{rsfs}{n}{it}{<-> rsfs10}{}
\DeclareMathAlphabet{\mathscr}{OT1}{rsfs}{n}{it}

\newcommand{\Z}{{\mathbb Z}}

\newcommand{\Q}{{\mathbb Q}}

\newcommand{\R}{{\mathbb R}}

\newcommand{\Ga}{\mathrm{Gal}}
\newtheorem{thm}{Theorem}[section]
\newtheorem{lemma}[thm]{Lemma}
\newtheorem{prop}[thm]{Proposition}

\begin{document}

\title[Generic elements]{Generic elements of a Zariski-dense subgroup form an open subset}

\author[G.~Prasad]{Gopal Prasad}

\author[A.S.~Rapinchuk]{Andrei S. Rapinchuk}
\vskip3mm
\maketitle

\centerline{\it Dedicated to E.B.Vinberg on his 80th birthday}
\vskip4mm

\begin{abstract}
Let $G$ be a semi-simple algebraic group over a finitely generated field $K$ of characteristic zero, and
let $\Gamma \subset G(K)$ be a finitely generated Zariski-dense subgroup. In this note we prove that the set of
$K$-\emph{generic elements} of $\Gamma$ (whose existence was established earlier in \cite{PR-generic}) is open in the profinite topology
of $\Gamma$. We then extend this result to the fields of positive characteristic, and also prove the existence of generic elements
in this case.
\end{abstract}

\section{Introduction}

This is a companion paper to \cite{PR-generic} where we first proved the existence of generic elements in an arbitrary Zariski-dense subgroup of the group of points of a semi-simple algebraic group over a finitely generated field of characteristic zero. Since then generic elements have been used in a variety of situations, in particular, to resolve some long-standing problems about isospectral locally symmetric spaces \cite{PR-IHES} (see also \cite{PR-MSRI} for a survey). This prompted us to try to understand the structure of the set of all generic elements in a given (finitely generated) Zariski-dense subgroup. The goal of this note is to establish a rather surprising fact that this set is open in the profinite topology of the subgroup -- see below for a more general/precise statement which applies to fields of any characteristic. We begin by recalling the relevant definitions.

\smallskip

Let $G$ be a semi-simple algebraic group over a field $K$. Fix a maximal $K$-torus $T$ of $G$, and let $\Phi(G , T)$ and $W(G , T)$ denote the corresponding root system and the Weyl group. The natural action of the absolute Galois group $\Ga(K^{\mathrm{sep}}/K)$, where $K^{\mathrm{sep}}$ is a fixed separable closure of $K$, on the character group $X(T)$ of $T$ gives rise to  a group homomorphism
$$
\theta_T \colon \Ga(K^{\mathrm{sep}}/K) \to \mathrm{Aut}(\Phi(G , T))
$$
that factors through the Galois group $\Ga(K_T/K)$ of the minimal splitting field $K_T$ of $T$ in $K^{\mathrm{sep}}$ inducing an \emph{injective} homomorphism
$\bar{\theta}_T \colon \Ga(K_T/K) \to \mathrm{Aut}(\Phi(G , T))$. We say that $T$ is \emph{generic} over $K$ if $\theta_T(\Ga(K^{\mathrm{sep}}/K)) \supset W(G , T)$. Furthermore, a regular semi-simple element $g \in G(K)$ is called $K$-{\it generic} if the $K$-torus $T = Z_G(g)^{\circ}$ (the connected component of the centralizer of $g$) is generic over $K$ (recall that $g \in T(K)$.) Some possible variations of this definition are discussed in \cite[9.4]{PR-MSRI}, but all the versions are equivalent for semi-simple elements without components of finite order (where the components are understood in terms of the decomposition $G = G_1 \cdots G_d$ as an almost direct  product of absolutely almost simple groups).

\smallskip

Now, fix a matrix $K$-realization $G \subset \mathrm{GL}_n$, and let $R$ be a subring of $K$. Quite often, by the \emph{congruence topology} on the group $G(R) := G(K) \cap {\rm{GL}}_n(R)$ one understands the topology having as a fundamental system of neighborhoods of the identity the family of congruence subgroups $$G(R , \mathfrak{a}) =: G(R) \cap {\rm{GL}}_n(R , \mathfrak{a}), \ \ \text{where} \ \  {\rm{GL}}_n(R , \mathfrak{a}) = \{ X \in {\rm{GL}}_n(R) \: \vert \: X \equiv I_n(\mathrm{mod}\: \mathfrak{a}) \},$$ in the obvious notation, for \emph{all} nonzero ideals $\mathfrak{a}$ of $R$. However, in this note we reserve this term for the (generally) weaker topology defined by the congruence subgroup $G(R , \mathfrak{a})$ where $\mathfrak{a} \subset R$ is an ideal of \emph{finite index}, that is the quotient $R/\mathfrak{a}$ is finite. Any such congruence subgroup is obviously a normal subgroup of finite index in $G(R)$, and consequently the topology induced by the congruence topology in this sense on any subgroup $\Gamma \subset G(R)$ is generally coarser than the \emph{profinite topology} of $\Gamma$
defined by \emph{all} normal subgroups $N \subset \Gamma$ of finite index.

\medskip

We can now formulate the main result.

\smallskip

\noindent {\bf Theorem 1.} {\it Let $G$ be a semi-simple algebraic group defined over a finitely generated field $K$ of any characteristic, $R \subset K$ be a finitely generated subring and $\Gamma \subset G(R)$ be a subgroup which is Zariski-dense in $G$. Then the set $\Delta(\Gamma , K)$ of regular semi-simple $K$-generic elements is open in $\Gamma$ in the congruence topology defined by ideals $\mathfrak{a} \subset R$ of finite index. In particular, $\Delta(\Gamma , K)$ is open in $\Gamma$ in the profinite topology.}

\medskip

\noindent {\bf Corollary.} {\it Let $G$ be a semi-simple algebraic group defined over a finitely generated field $K$, and let $\Gamma \subset G(K)$ be a finitely generated Zariski-dense subgroup. Then the set $\Delta(\Gamma , K)$ is open in $\Gamma$ in the profinite topology.}

\medskip

The proof of Theorem 1 requires a suitable generalization of Chebotarev's Density Theorem, and in \S\ref{S:Cheb} we give this generalization for fields of characteristic zero  - see Proposition \ref{P:1}. Then in \S\ref{S:Proof} we combine this proposition with some techniques developed earlier in \cite{PR-generic} to prove Theorem 1 in characteristic zero. The argument easily extends to positive characteristic provided that one can generalize Chebotarev's Theorem to this case; we establish the required generalization in  \S\ref{S:posit}.

While Theorem 1 gives the \emph{openness} of the set $\Delta(\Gamma , K)$ of regular semi-simple $K$-generic elements in the profinite topology of a finitely generated Zariski-dense subgroup $\Gamma \subset G(K)$, its proof does not automatically yield the \emph{non-emptiness} of $\Delta(\Gamma , K)$. As we already pointed out, the existence of generic elements was first established in \cite{PR-generic} over fields of characteristic zero - for the reader's convenience we summarize this argument in Remark 3.2. One of its essential components is a form of weak approximation that asserts that for a finite set $V$ of discrete valuations of $K$ that is constructed in the proof, the closure of the image of the diagonal embedding $\Gamma \hookrightarrow G_V = \prod_{v \in V} G(K_v)$ is open. In characteristic zero $V$ is selected so that the completions $K_v$ for $v \in V$ are the $p$-adic fields $\Q_p$ for pairwise distinct primes $p$, and then the required openness is an easy consequence of the Zariski-density of $\Gamma$. In characteristic $p > 0$, this property becomes significantly more delicate. If $p > 3$ then the argument from  \cite{PR-generic} can still be made to work using the strong approximation theorem of B.~Weisfeiler \cite{Weis}. Additional complications in characteristic 2 and 3 come from so-called {\it exceptional isogenies}, whose existence render the openness statement invalid even after passing to the universal cover. For this reason, the existence of generic elements over fields of positive characteristic remained unproven until recently. J.\,Schwartz in his dissertation \cite{Schwartz} used the approximation results of R.~Pink \cite{Pink-CS}, \cite{Pink-SA} to prove the existence of generic elements, in particular, over global fields of all characteristics. Here we will establish the following general result for all absolutely almost simple groups.

\smallskip

\noindent {\bf Theorem 2.} {\it Let $G$ be a connected absolutely almost simple algebraic group over a finitely generated field $K$ (of any characteristic), and let $\Gamma'$ be a Zariski-dense subsemigroup of $G(K)$ that contains an element of infinite order.  Then $\Gamma'$ contains a regular semisimple element $\gamma' \in \Gamma'$ of infinite order that is $K$-generic.}

\smallskip

We will prove Theorem 2 in \S\ref{S:Exist}. The argument heavily relies on the results of Pink \cite{Pink-CS} which we will briefly review below for the reader's convenience.
\smallskip

\noindent{\bf Notation}.  Given a variety $X$ defined over a field $K$, and a field extension $L$ of $K$, we will denote by $X_L$, or $(X)_L$, the $L$-variety obtained from $X$ by base change $K\hookrightarrow L$. If $X$ is an algebraic  $K$-group, by $L$-torus of $X$, we will mean a $L$-torus of $X_L$.

\medskip

\section{A generalization of Chebotarev's Density Theorem}\label{S:Cheb}

\begin{prop}\label{P:1}
Let $R$ be a finitely generated subring of a finitely generated field $K$ of characteristic zero, and let $L$ be a finite Galois extension of $K$ with Galois group $\mathscr{G} = \Ga(L/K)$.   Fix a conjugacy class $\mathscr{C}$ of $\mathscr{G}$. Then there exists an infinite set of primes $\Pi$ such that for each $p \in \Pi$ there exists an embedding $\iota_p \colon K \hookrightarrow \Q_p$ with the following properties:

\smallskip

{\rm (1)} $\iota_p(R) \subset \Z_p$,

\smallskip

{\rm (2)} \parbox[t]{14.5cm}{if $v$ denotes the discrete valuation of $K$ obtained by pulling back the $p$-adic valuation of $\Q_p$ (so that $K_v = \Q_p$), then any extension $w \vert v$ to $L$ is unramified and the Frobenius automorphism of $L_w/K_v$ belongs to $\mathscr{C}$.}
\end{prop}
\begin{proof}
In this argument, for a (monic) polynomial $f(x) \in A[x]$ over an integral domain $A$ we let $\delta_f \in A$ denote its discriminant.  Without loss of generality, we may assume that $K$ is the field of fractions of $R$. Using Noether's Normalization Theorem, we can find algebraically independent $t_1, \ldots , t_r \in \Q R$ so that $\Q R$ is integral over $\Q[t_1, \ldots , t_r]$. In fact, we may further assume $t_1, \ldots , t_r \in R$ and then pick a finite set of primes $S$ such that for  $p_S := \prod_{p\in S} p$, the localization $\Z_S= \Z[p^{-1}_S]$ of $\Z$ away from $S$, we have that $\Z_S R$ is integral over $\Z_S[t_1, \ldots , t_r]$. Now, set $k = \Q(t_1, \ldots t_r)$ and let $E$ denote the Galois closure of $L$ over $k$ with Galois group $\mathscr{H} = \Ga(E/k)$. We can pick a primitive element $\alpha$ for $E$ over $k$ whose minimal polynomial $f$ is of the form
$$
f(x) = x^n + z_{n-1}(t_1, \ldots , t_r)x^{n-1} + \cdots + z_0(t_1, \ldots , t_r)
$$
with $z_i(t_1, \ldots , t_r) \in \Z[t_1, \ldots , t_r]$. Now, fix $\sigma \in \mathscr{C}$, and let $\widetilde{\sigma} \in \mathscr{H}$ be an automorphism that acts trivially on $K$ and restricts to $\sigma$ on $L$. Pick polynomials $g_0, \ldots , g_{n-1}$ and $h \in \Z_S[t_1, \ldots , t_r]$ so that
\begin{equation}\label{E:1}
\widetilde{\sigma}(\alpha) = \sum_{j = 0}^{n-1} c_j \alpha^j \ \ \text{where} \ \ c_j = g_j/h.
\end{equation}
Using Hilbert's Irreducibility Theorem (cf.\,\cite[Ch.\,3]{Se-Inv} and references therein), we can find $(a^0_1, \ldots , a^0_r) \in \Q^r$ such that $h(a^0_1, \ldots , a^0_r) \neq 0$ and the polynomial
$$
f_0(x) := x^n + z_{n-1}(a^0_1, \ldots , a^0_r)x^{n-1} + \cdots + z_0(a^0_1, \ldots , a^0_r) \in \Q[x]
$$
is irreducible. Then, if we write the discriminant $\delta_f$ as a polynomial in $t_1, \ldots , t_r$, we automatically have $\delta_f(a^0_1, \ldots , a^0_r) \neq 0$.
Let $\mathfrak{m}_0$ be the maximal ideal of $\Q[t_1, \ldots , t_r]$ generated by $t_1 - a^0_1, \ldots , t_r - a^0_r$, let $A$ be the corresponding local ring
$\Q[t_1, \ldots , t_r]_{\mathfrak{m}_0}$ with the maximal ideal $\mathfrak{m} = \mathfrak{m}_0A$, and let $B$ be the integral closure of $A$ in $E$. By construction, $\alpha \in B$ and the discriminant of the basis $1, \alpha, \ldots , \alpha^{n-1}$ is a unit in $A$. So, a standard argument using traces  shows that in fact $B = A[\alpha]$. The specialization homomorphism $\psi \colon A \to \Q$ with kernel $\mathfrak{m}$ that sends $t_i$ to $a^0_i$ for $i = 1, \ldots , r$, extends to a homomorphism $\widetilde{\psi} \colon B \to \overline{\Q}$ into the algebraic closure of $\Q$ (see \cite[Ch.\,VII, Proposition 3.1]{Lang}); note that $\mathfrak{M} := \mathrm{Ker}\: \widetilde{\psi}$ is a maximal ideal of $B$ lying above $\mathfrak{m}$.  Let $E_0 = \widetilde{\psi}(B) \simeq B/\mathfrak{M}$; clearly, $E_0 = \Q(\alpha_0)$ where $\alpha_0 = \widetilde{\psi}(\alpha)$ is a root of $f_0$. Since $f_0$ is irreducible, we have
\begin{equation}\label{E:deg}
[E_0 : \Q] = n = [E : k].
\end{equation}
Furthermore, by \cite[Ch.\,VII, Proposition 2.5]{Lang}, the extension $E_0/\Q$ is normal and if we let $\mathscr{H}(\mathfrak{M})$ denote the decomposition subgroup of $\mathfrak{M}$ in $\mathscr{H}$, then the reduction of automorphism modulo $\mathfrak{M}$ yields a \emph{surjective} homomorphism
$$
\rho \colon \mathscr{H}(\mathfrak{M}) \to \Ga(E_0/\Q) =: \mathscr{H}_0.
$$
Since according to (\ref{E:deg}) we have $| \mathscr{H} | = | \mathscr{H}_0 |$, we conclude that $\mathscr{H}(\mathfrak{M}) = \mathscr{H}$, and $\rho \colon \mathscr{H} \to \mathscr{H}_0$ is actually an isomorphism. Let $\widetilde{\sigma}_0 := \rho(\widetilde{\sigma}) \in \mathscr{H}_0$.

Enlarging $S$ if necessary, we may assume that $a^0_1, \ldots , a^0_r \in \Z_S$ and  $\delta_f(a^0_1, \ldots , a^0_r)$ and $h(a^0_1, \ldots , a^0_r)$ are $p$-adic units for all primes $p \notin S$. Let $\Pi$ be the set of all primes $p \notin S$ such that the extension $E_0/\Q$ is unramified at $p$ and for a suitable extension $u$ of the $p$-adic place to $E_0$, the Frobenius automorphism $\mathrm{Fr}(u \vert p)$ of $(E_0)_u/\Q_p$ is $\widetilde{\sigma}_0$. By Chebotarev's Density Theorem (cf.\,\cite[Ch.\,VII, 2.4]{ANT}), the set $\Pi$ is infinite, and we will show that it is as required.

Let $p \in \Pi$. By our construction, we can then pick an extension of the $p$-adic valuation $u$ to $E_0$ such that $(E_0)_u/\Q_p$ is unramified with the Frobenius automorphism $\mathrm{Fr}(u \vert p)$ equal to $\widetilde{\sigma}_0$. Then $u$ corresponds to an embedding $\varepsilon \colon E_0 \hookrightarrow \overline{\Q_p}$ into the algebraic closure of $\Q_p$, and we set $\mathcal{E} = \Q_p(\varepsilon(\alpha_0))$ (clearly, $\mathcal{E}$ is naturally identified with the completion $(E_0)_u$) and let $\varphi$ be the Frobenius automorphism of $\mathcal{E}/\Q_p$. Let $\mathcal{O}$ (resp., $\mathfrak{p}$) be the valuation ring (resp., valuation ideal) in $\mathcal{E}$.

Now,  pick $a^1_i \in a^0_i + p\Z_p$ for $i = 1, \ldots , r$ so that $a^1_1, \ldots , a^1_r$ are algebraically independent over $\Q$, and then let
$$
f_1(x) := x^n + z_{n-1}(a^1_1, \ldots , a^1_r)x^{n-1} + \cdots + z_0(a^1_1, \ldots , a^1_r) \in \Z_p[x].
$$
Note that $f_1(x) \equiv f_0(x) (\mathrm{mod}\: p)$, so $\delta_{f_1} \equiv \delta_{f_0} (\mathrm{mod}\: p)$ and therefore $\delta_{f_1} \not\equiv 0(\mathrm{mod}\: p)$.
By construction $f_0(\varepsilon(\alpha_0)) = 0$, and consequently $f_1(\varepsilon(\alpha_0)) \equiv 0 (\mathrm{mod}\: \mathfrak{p})$. Since $\delta_{f_1} \not\equiv 0(\mathrm{mod}\: p)$, we have $f'_1(\varepsilon(\alpha_0)) \not\equiv 0(\mathrm{mod}\: \mathfrak{p})$, and therefore by Hensel's Lemma, there exists  a root $\alpha_1 \in \mathcal{O}$ of $f_1$ such that $\alpha_1 \equiv \varepsilon(\alpha_0) (\mathrm{mod}\: \mathfrak{p})$. We note that since $\delta_{f_0} \not\equiv 0(\mathrm{mod}\: \mathfrak{p})$, we have $\mathcal{O} = \Z_p[\varepsilon(\alpha_0)]$ (cf.\,\cite[Ch.\,I, Theorem 7.5]{Jan}), and therefore the residue field of $\mathcal{E}$ is generated over the prime subfield $\mathbf{F}_p$ by the image $\overline{\varepsilon(\alpha_0)} = \overline{\alpha_1}$. Since $\mathcal{E}/\Q_p$ is unramified, it follows that $\mathcal{E} = \Q_p(\alpha_1)$.  Since $a^1_1, \ldots , a^1_r$ are algebraically independent over $\Q$, there is an embedding $k \hookrightarrow \Q_p$ sending $t_i$ to $a^1_i$ for $i = 1, \ldots , r$. This embedding extends to a (dense) embedding $\iota \colon E \hookrightarrow \mathcal{E}$ sending $\alpha$ to $\alpha_1$.

\medskip

\noindent {\bf Claim.} {\it For $a \in E$, we have $\iota(\widetilde{\sigma}(a)) = \varphi(\iota(a))$.}

\medskip

Indeed, it is enough to prove this for $a = \alpha$. It follows from (\ref{E:1}) that
\begin{equation}\label{E:110}
\iota(\widetilde{\sigma}(\alpha)) = \sum_{j = 0}^{n-1} c_j(a^1_1, \ldots , a^1_r) (\alpha_1)^j.
\end{equation}
Clearly, we have $h(a^1_1, \ldots , a^1_r) \equiv h(a^0_1, \ldots , a^0_r)(\mathrm{mod}\: p)$, and in particular, $h(a^1_1, \ldots , a^1_r)$ is a $p$-adic unit. It follows that for all $j = 0, \ldots , n-1$, the elements $c_j(a^0_1, \ldots , a^0_r)$ and $c_j(a^1_1, \ldots , a^1_r)$ both lie in $\Z_p$ and are congruent modulo $p$. Applying to (\ref{E:1}) the specialization map $\widetilde{\psi}$, we obtain that
\begin{equation}\label{E:111}
\widetilde{\sigma}_0(\alpha_0) = \sum_{j = 0}^{n-1} c_j(a^0_1, \ldots , a^0_r) (\alpha_0)^j.
\end{equation}
Since $\alpha_1 \equiv \varepsilon(\alpha_0)(\mathrm{mod}\: \mathfrak{p})$, we have
$$
\varphi(\iota(\alpha)) = \varphi(\alpha_1) \equiv \varphi(\varepsilon(\alpha_0)) (\mathrm{mod}\: \mathfrak{p}).
$$
On the other hand, since $\mathrm{Fr}(u \vert p) = \widetilde{\sigma}_0$, we have $\varphi(\varepsilon(\alpha_0)) = \varepsilon(\widetilde{\sigma}_0(\alpha_0))$. Combining this with (\ref{E:110}), (\ref{E:111}) and the fact that $\alpha_1 \equiv \varepsilon(\alpha_0) (\mathrm{mod}\: \mathfrak{p})$, we conclude that
\begin{equation}\label{E:2}
\varphi(\iota(\alpha))  \equiv \iota(\widetilde{\sigma}(\alpha)) (\mathrm{mod}\: \mathfrak{p}).
\end{equation}
Now, both $\iota(\widetilde{\sigma}(\alpha))$ and $\varphi(\iota(\alpha))$ are roots of $f_1(x)$. But since $\delta_{f_1} \not\equiv 0$, the polynomial $f_1$ has no multiple roots modulo $\mathfrak{p}$, so (\ref{E:2}) implies that $\iota(\widetilde{\sigma}(\alpha)) = \varphi(\iota(\alpha))$ as required.

\medskip

By construction, $\widetilde{\sigma}$ acts on $K$ trivially. So, it follows from the above claim that $\iota(K) \subset \mathcal{E}^{\varphi} = \Q_p$ because $\varphi$ generates the Galois group $\Ga(\mathcal{E}/\Q_p)$. Since $\iota(\Z_S[t_1, \ldots , t_r]) \subset \Z_p$ and $\Z_S R$ is integral over $\Z_S[t_1, \ldots , t_r]$, we conclude that $\iota(R) \subset \Z_p$, so for $\iota_p$ we can take the restriction of $\iota$ to $K$. It follows from our construction that if we let $w_0$ denote the pullback to $L$ of the extension of the $p$-adic valuation to $\mathcal{E}$ so that the completion $L_{w_0}$ can be identified with the compositum $\iota(L)\Q_p$ inside $\mathcal{E}$, then $L_{w_0}/K_v$ is unramified and with the Frobenius automorphism  $\mathrm{Fr}(w_0 \vert v) = \sigma$. Since $L/K$ is a Galois extension, we conclude that \emph{any} extension $w \vert v$ is unramified and the Frobenius $\mathrm{Fr}(w \vert v)$ belongs to the conjugacy class $\mathscr{C}$.\end{proof}

\noindent {\bf Remark 2.2.} Proposition \ref{P:1} can be derived from a generalization of  Chebotarev's theorem to the case of schemes of finite type over $\Z$ (see \cite[Ch.\,9]{Se-NX(p)} and references therein). The above argument, however, shows a way to bypass this (rather technical) generalization and obtain the required proposition using only the classical form of Chebotarev's theorem.

\vskip5mm

\section{Proof of Theorem 1}\label{S:Proof}

\begin{lemma}\label{L:1}
Let $G$ be a semi-simple algebraic group over a field $\mathcal{K}$ which is complete with respect to a discrete valuation $v$. Fix a maximal $\mathcal{K}$-torus $T$ of $G$, let $T_{\mathrm{reg}}$ denote the Zariski-open set of regular elements, and consider the regular map
$$
\psi \colon G \times T_{\mathrm{reg}} \longrightarrow G, \ \ (g , t) \mapsto gtg^{-1}.
$$
Then the map  $\psi_{\mathcal{K}} \colon G(\mathcal{K}) \times T_{\mathrm{reg}}(\mathcal{K}) \to G(\mathcal{K})$ induced by $\psi$ on $\mathcal{K}$-points is open for the topology defined by $v$.
\end{lemma}

Indeed, a direct computation shows that the differential $d_{(g , t)} \psi$ is surjective for any $(g , t) \in G \times T_{\mathrm{reg}}$. So, our assertion follows from the Inverse Function Theorem (cf. \cite[\S 3.1]{PlR}, \cite[Part II, Ch.\,III]{Se-Lie}).

\smallskip

We will now recall one construction introduced in \cite{PR-generic}. Let $G$ be a semi-simple $K$-group, and let $T_1$ and $T_2$ be two maximal tori of $G$ defined over some extension $F/K$. Then there exists $g \in G(\overline{F})$ such that $T_2 = \iota_g(T_1)$, where $\iota_g(x) = gxg^{-1}$. Then $\iota_g$ induces an isomorphism between the Weyl groups $W(G , T_1)$ and $W(G , T_2)$. A different choice of $g$ will change this isomorphism by an inner automorphism of the Weyl group, implying that there is a \emph{canonical bijection} between the sets $[W(G , T_1)]$ and $[W(G , T_2)]$ of conjugacy classes in the respective groups; we will denote this bijection by $\iota_{T_1 , T_2}$. Moreover, if we let $\iota^*_g \colon {\rm{X}}(T_2) \to {\rm{X}}(T_1)$ denote the corresponding isomorphism of the character groups, then $\iota^*_g$ takes $\Phi(G , T_2)$ to $\Phi(G , T_1)$, and if we identify $\mathrm{Aut}(\Phi(G , T_1))$ with $\mathrm{Aut}(\Phi(G , T_2))$ using $\iota^{\natural}_g \colon \alpha \mapsto (\iota^*_g)^{-1} \circ \alpha \circ \iota^*_g$, for $\alpha\in\mathrm{Aut}(\Phi(G , T_1))$,  then the following holds:
{\it if $g \in G(E)$, where $E$ is an extension of $F$, then for any $\sigma \in \Ga(\overline{E}/E)$ we have
\begin{equation}\label{E:KKK}
  \iota^{\natural}_g(\theta_{T_1}(\sigma)) = \theta_{T_2}(\sigma)
\end{equation}
in the above notations.}

\medskip

{\it Proof of the theorem.} It is enough to show that the set $\Delta(\Gamma , K)$ of regular semi-simple $K$-generic elements in $\Gamma = G(R)$ is open in $\Gamma$ in the congruence topology defined by ideals $\mathfrak{a} \subset R$ of finite index. Let $g_0 \in \Delta(\Gamma , K)$. Then for the maximal $K$-torus $T_0 = Z(g_0)^{\circ}$ we have the inclusion
$$
\bar{\theta}_{T_0}(\Ga(K_{T_0}/K)) \supset W(G , T_0)
$$
in the above notations. Let $w_1, \ldots , w_r$ be a set representative of all nontrivial conjugacy classes of $W(G , T_0)$,  let $\widetilde{\sigma}_i \in \Ga(\overline{K}/K)$ be such that that $\theta_{T_0}(\widetilde{\sigma}_i) = w_i$ for $i = 1, \ldots , r$, and let $\sigma_i$ be the image of $\widetilde{\sigma}_i$
in $\Ga(K_{T_0}/K)$ (so that $\bar{\theta}_{T_0}(\sigma_i) = w_i$).  Applying Proposition \ref{P:1} to $L = K_{T_0}$, we can find $r$ distinct primes $p_1, \ldots , p_r$ such that for each $i \in \{1, \ldots , r\}$ there is an embedding $\iota_{p_i} \colon K \hookrightarrow \Q_{p_i}$ such that $\iota_{p_i}(R) \subset \Z_{p_i}$ and for a suitable extension $u_i \vert v_{p_i}$, where $v_{p_i}$ is the pullback of the $p_i$-adic valuation on $\Q_{p_i}$, the extension $L_{u_i}$ of $K_{v_{p_i}} = \Q_{p_i}$ is unramified with the Frobenius automorphism $\sigma_i$. According to Lemma \ref{L:1}, for each $i \in \{1, \ldots , r\}$, the set
$$
\mathcal{U}_{p_i} := \psi_{\Q_{p_i}}(G(\Q_{p_i}) \times (T_0)_{\mathrm{reg}}(\Q_{p_i})) \ \ \text{where} \ \ \psi_{\Q_{p_i}}(g , t) = gtg^{-1}.
$$
is open and obviously contains $g_0$. So, we can find $\ell_i \geqslant 1$ such that the coset $g_0G(\Z_{p_i} , p_i^{\ell_i}\Z_{p_i})$ of the corresponding congruence subgroup is contained in $\mathcal{U}_{p_i}$. Set
$$
\mathfrak{a} = \bigcap_{i = 1}^r (R \cap \iota_{p_i}^{-1}(p_i^{\ell_i}\Z_{p_i})).
$$
Clearly, $\mathfrak{a}$ is an ideal of $R$ having finite index, and to conclude the proof we will show that $g_0\Gamma(\mathfrak{a}) \subset \Delta(\Gamma , K)$.  Let $g \in g_0 \Gamma(\mathfrak{a})$. Then by construction $g \in \mathcal{U}_{p_i}$ for all $i = 1, \ldots , r$, and in particular $g$ is a regular semi-simple element. Furthermore, if $T = Z_{G}(g)^{\circ}$, then for each $i = 1, \ldots , r$ there exists $g_i \in G(\Q_{p_i})$ such that $\iota_{g_i}(T_0) = T$. It follows from (\ref{E:KKK}) that $\iota^{\natural}_{g_i}(\theta_{T_0}(\widetilde{\sigma}_i)) = \theta_{T}(\widetilde{\sigma}_i)$, and therefore the conjugacy class $\iota_{T_0 , T}([w_i])$ of $W(G , T)$ intersects $\theta_T(\Ga(\overline{\Q_{p_i}}/\Q_{p_i})) \subset \theta_T(\Ga(\overline{K}/K))$. This being true for each $i = 1, \ldots , r$, we conclude that the subgroup $\theta_T(\Ga(\overline{K}/K)) \cap W(G , T)$ intersects every conjugacy class of $W(G , T)$. Applying an elementary fact (Jordan's theorem) from group theory, we obtain that $\theta_T(\Ga(\overline{K}/K)) \supset W(G , T)$, as required. \hfill $\Box$

\vskip1mm

\noindent {\bf Remark 3.2.} We would like to indicate that the above argument is parallel to the argument developed in \cite{PR-generic} in order to prove the existence of $K$-generic elements in any Zariski-dense subgroup (and in fact those with special properties such as $\R$-regularity if $K \subset \R$). This indicates that the construction of generic elements from \cite{PR-generic} in fact enables one to obtain {\it all of them}. More precisely, fix a $K$-torus $T_0$ of $G$ and as above let $[w_1], \ldots , [w_r]$ be all nontrivial conjugacy classes of $W(G , T_0)$. It follows from Proposition \ref{P:1} that one can pick $r$ distinct primes $p_1, \ldots , p_r$ such that for each $i = 1, \ldots , r$ there exists an embedding $\iota_i \colon K \hookrightarrow \Q_{p_i}$ such that $\iota_i(R) \subset \Z_{p_i}$ and $G$ splits over $\Q_{p_i}$. Then one shows that there is a maximal $\Q_{p_i}$-torus $T_i$ of $G/\Q_{p_i}$ such that
$$
\theta_{T_i}(\Ga(\overline{\Q_{p_i}}/\Q_{p_i})) \cap \iota_{T_0 , T_i}([w_i]) \neq \varnothing.
$$
Let
$$
\mathcal{U}_{p_i} = \psi_{\Q_{p_i}}(G(\Q_{p_i}) \times (T_i)_{\mathrm{reg}}(\Q_{p_i})) \ \ \text{where} \ \ \psi_{\Q_{p_i}}(g , t) = gtg^{-1}.
$$
We observe that $\mathcal{U}_{p_i}$ intersects every open subgroup of $G(\Q_{p_i})$. Since the $p_i$ are distinct, a standard approximation argument shows that since $\Gamma$ is Zariski-dense, its closure in $\prod_{i = 1}^r G(\Q_{p_i})$ is open, and therefore $\Gamma \cap \prod_{i = 1}^r \mathcal{U}_{p_i} \neq \varnothing$. Then the argument used in the proof of the above theorem shows that any element of this intersection is generic over $K$.

\vskip5mm

\section{Positive characteristic case: Theorem 1}\label{S:posit}

The argument given in \S\ref{S:Proof} is independent of the characteristic of the base field. So, to prove Theorem 1 in positive characteristic
we only need to provide a suitable analogue of Proposition \ref{P:1}. It suffices to prove the following.
\begin{prop}\label{P:2}
Let $R$ be a finitely generated subring of  an infinite finitely generated field $K$ of characteristic $p > 0$, and let $L/K$ be a finite Galois extension with Galois group $\mathscr{G} = \Ga(L/K)$. Fix a conjugacy class $\mathscr{C}$ of $\mathscr{G}$. Then there exists a (nontrivial) discrete valuation $v$ on $K$ such that the completion $K_v$ is locally compact, $R$ lies in the corresponding valuation ring $\mathcal{O}_v$, and for any extension $w \vert v$, the Frobenius automorphism of $L_w/K_v$ belongs to $\mathscr{C}$.
\end{prop}

We will only indicate the changes that need to be made in the proof of Proposition \ref{P:1}. Again, we may (and we will) assume that the field of fractions of $R$ coincides with $K$, and then  find in $R$ a separable transcendence basis $t_0, \ldots , t_r$ for $K$ over the prime subfield $\mathbb{F}_p$ (which means that $K$ is a finite separable extension of $k := \mathbb{F}_p(t_0, \ldots , t_r)$).    Let $E$ denote the Galois closure of $L$ over $k$ with Galois group $\mathscr{H} = \Ga(E/k)$. Set $A = \mathbb{F}_p[t_0]$ and $k_0 = \mathbb{F}_p(t_0)$, and  pick a primitive element $\alpha \in E$ over $k$ whose minimal polynomial is of the form
$$
f(x) = x^n + z_{n-1}(t_1, \ldots , t_r)x^{n-1} + \cdots + z_0(t_1, \ldots , t_r)
$$
where $z_i(t_1, \ldots , t_r) \in A[t_1, \ldots , t_r]$. We can find $h_1 \in A[t_1, \ldots, t_r]$ so that the extension of the corresponding localizations $R_{h_1} / A[t_1, \ldots , t_r]_{h_1}$ is integral. Next, pick a representative $\sigma \in \mathscr{C}$ and let $\widetilde{\sigma} \in \mathscr{H}$ be such that $\widetilde{\sigma} \vert L = \sigma$. There exist $g_0, \ldots , g_{n-1}$ and $h_2 \in A[t_1, \ldots , t_r]$ such that
$$
\widetilde{\sigma}(\alpha) = \sum_{j = 0}^{n-1} c_j \alpha^j \ \ \text{where} \ \ c_j = g_j/h_2.
$$
Set $h = h_1 h_2$. By Hilbert's Irreducibility Theorem, one can find  $(a^0_1, \ldots, a^0_r) \in (k_0)^{r}$ such that $h(a^0_1, \ldots , a^0_r) \cdot \delta_f(a^0_1, \ldots , a^0_r) \neq 0$, where $\delta_f\in A[t_1,\ldots, t_r]$ is the discriminant of $f$, and the polynomial
$$
f_0(x) = x^n + z_{n-1}(a^0_1, \ldots , a^0_r)x^{n-1} + \cdots + z_0(a^0_1, \ldots , a^0_r) \in k_0[x]
$$
is irreducible. We can find a finite set of places $S$ of $k_0$, that includes the place at infinity, such that for any place $v \notin S$ and the corresponding valuation ring $\mathcal{O}_{k_0 , v}$, we have the inclusions
$$
a^0_1, \ldots , a^0_r \in \mathcal{O}_{k_0 , v} \ \ \text{and} \ \  h(a^0_1, \ldots , a^0_r) \, , \, \delta_f(a^0_1, \ldots , a^0_r) \in \mathcal{O}_{k_0 , v}^{\times}.
$$
We then consider the extension $E_0 = k_0(\alpha_0)$ where $\alpha_0$ is a root of $f_0$. As in the proof of Proposition \ref{P:1}, we see that
$E_0/k_0$ is a Galois extension such that the specialization $t_i \mapsto a^0_i$ for $i = 1, \ldots , r$ yields a natural isomorphism between $\mathscr{H}$ and the Galois group $\mathscr{H}_0 = \Ga(E_0/k_0)$.  We let $\widetilde{\sigma}_0 \in \Ga(E_0/k_0)$ denote the automorphism corresponding to $\widetilde{\sigma}$ under this isomorphism. Applying Chebotarev's Density Theorem, we find a place $v_0 \notin S$ of $k_0$ which is unramified in $E_0$ and such that for a suitable extension $w_0$, the Frobenius automorphism $\mathrm{Fr}(w_0 \vert v_0)$ is $\widetilde{\sigma}_0$. The valuation $w_0$ corresponds to an embedding $\varepsilon \colon E_0 \hookrightarrow ((k_0)_{v_0})^{\mathrm{sep}}$ into the separable closure of the completion $(k_0)_{v_0}$, and we set $\mathcal{E} = (k_0)_{v_0}(\varepsilon(\alpha_0))$ observing that it is naturally identified with the completion $(E_0)_{w_0}$.

Let $\mathcal{O}_0$ be the valuation ring in $(k_0)_{v_0}$, and let $\mathfrak{p}_0$ be its maximal ideal. Since the latter is uncountable, we can find $a_i^1 \in a_i^0 + \mathfrak{p}_0$ for $i = 1, \ldots , r$ so that the elements $t_0, a_1^1, \ldots , a_r^1 \in (k_0)_{v_0}$ are algebraically independent over $\mathbb{F}_p$. Then there is an embedding $\iota_0 \colon k \hookrightarrow (k_0)_{v_0}$ that sends $t_0$ to $t_0$ and $t_i$ to $a_i^1$ for $i = 1, \ldots , r$. Consider the polynomial
$$
f_1(x) := x^n + z_{n-1}(a_1^1, \ldots , a_r^1) x^{n-1} + \cdots + z_0(a_1^1, \ldots , a_r^1) \in \mathcal{O}_0[x].
$$
Applying Hensel's Lemma as in the proof of Proposition \ref{P:1}, we see that there is a root $\alpha_1$ of $f_1$ in the valuation ring $\mathcal{O}(\mathcal{E})$ of $\mathcal{E}$ which is congruent to $\varepsilon(\alpha_0)$ modulo the corresponding valuation ideal. Then $\iota_0$ extends to a dense embedding $\widetilde{\iota} \colon E \hookrightarrow \mathcal{E}$ that sends $\alpha$ to $\alpha_1$. Let $\iota$ be the restriction of $\widetilde{\iota}$ to $K$, and let $v$ be the pullback of  $w_0$ to $K$. Then the completion $K_v$ can be identified with the compositum $\iota(K)(k_0)_{v_0}$ inside $\mathcal{E}$, hence it is locally compact. It follows from our construction that $\iota(A[t_1, \ldots , t_r]) \subset \mathcal{O}_v$ (= the valuation ring of $K_v$) and $h_1(a_1^1, \ldots , a_r^1) \in \mathcal{O}_v^{\times}$. Since the ring extension $R_{h_1}/A[t_1, \ldots, t_r]_{h_1}$ is integral, we conclude that $\iota(R) \subset \mathcal{O}_v$. Finally, repeating verbatim the argument given in the proof of Proposition \ref{P:1}, we see that if $w$ is the pullback to $L$ of the valuation on $\mathcal{E}$ with respect to the restriction $\widetilde{\iota} \vert L$, then $L_w/K_v$ is unramified and the Frobenius automorphism $\mathrm{Fr}(w \vert v)$ is $\sigma$. Now, the fact that $L/K$ is a Galois extension implies that \emph{any} extension $w \vert v$ is unramified with $\mathrm{Fr}(w \vert v)$ belonging to the conjugacy class $\mathscr{C}$, as required.

\smallskip

\noindent {\bf Remark 4.2.} The proof of Proposition \ref{P:2} actually gives an infinite number of inequivalent valuations $v$ having the properties indicated in the statement.

\section{Positive characteristic: existence of generic elements}\label{S:Exist}

The goal of this section is to prove Theorem 2 (of the introduction). The argument relies heavily on the results of Pink \cite{Pink-CS} which we briefly summarize below. But first we would like to reduce the proof to the case where $\Gamma'$ is finitely generated. We recall that an abstract semigroup is called \emph{locally finite} if every finitely generated subsemigroup of it is finite.
\begin{lemma}\label{L:Exist1}
Let $G$ be an absolutely almost simple algebraic $K$-group, and let $\Gamma'$ be a Zariski-dense subsemigroup of $G(K)$. If $\Gamma'$ is not locally finite, then it contains a finitely generated Zariski-dense subsemigroup $\Delta'$ which is also Zariski-dense.
\end{lemma}
\begin{proof}
Pick a finitely generated subsemigroup $\Delta' \subset \Gamma'$ for which the Zariski closure $H = \overline{\Delta'}$ has maximum possible dimension. We note that the Zariski closure of a subsemigroup is actually a  subgroup, and  since $\Gamma'$ is not locally finite, $H$ is of positive dimension. Take any $\gamma' \in \Gamma'$, and let $H'$ be the Zariski closure of the subsemigroup generated by $\Delta'$ and $\gamma'$. By construction, $\dim H = \dim H'$, and therefore the connected components $H^{\circ}$ and $(H')^{\circ}$ coincide. It follows that $\gamma'$ normalizes $H^{\circ}$. Since $\gamma' \in \Gamma'$ is arbitrary, $H^{\circ}$ is normalized by all of $\Gamma'$, hence by $\overline{\Gamma'} = G$. Now, since $\dim H^{\circ} > 0$,  the fact that $G$ is absolutely almost simple implies that $H^{\circ} = G$, so $\Delta'$ is Zariski-dense.
\end{proof}

Since a semigroup containing an element of infinite order is not locally finite, it follows from the lemma that $\Gamma'$ as in Theorem 2 always contains a finitely generated subsemigroup $\Delta'$ which is Zariski-dense in $G$. Then it is enough to establish the existence of $K$-generic semi-simple elements of infinite order in $\Delta'$. {\em Thus, we may $($and we will$)$ henceforth assume that $\Gamma'$ is finitely generated and will denote by $\Gamma$ the subgroup of $G(K)$ generated by $\Gamma'$}.
\vskip1mm

Before we proceed with the proof of the theorem, we would like to point out that a theorem due to Schur (cf.\,\cite[Theorem 9.9]{Lam}) implies that the condition that $\Gamma'$ is not locally finite is in fact equivalent to the condition that it contains an element of infinite order. (Technically, Schur's theorem treats linear {\it groups}, so we note that a semigroup consisting of elements of finite order is automatically a group.)

\vskip3mm

\noindent {\bf On Pinks' approximation results.} In this subsection, we will review the notions involved in the results of Pink \cite{Pink-CS}, \cite{Pink-SA} and then give a precise statement of his main result in \cite{Pink-CS} in the form needed for our purpose. For $i = 1, \ldots , r$, let $G_i$ be a connected  absolutely simple adjoint group over a local (i.e.\,nondiscrete locally compact) field $F_i$. Let $G$ denote the group scheme over the commutative semi-simple ring $F = \prod_{i = 1}^r F_i$ with fibers $G_i$ so that $G(F) = \prod_{i = 1}^r G_i(F_i)$. Let $\Gamma \subset G(F)$ be a subgroup with compact closure and with Zariski-dense projections in all factors. Following Pink \cite{Pink-CS}, we say that a triple $(E, H, \varphi)$ consisting of a closed semisimple subring $E \subset F$ such that $F$ is a module of finite type over $E$,  a group scheme $H$ over $E$ whose fibers over factor fields of $E$ are  connected absolutely simple adjoint groups, and an isogeny $\varphi \colon H \times_E F \to G$ such that $\Gamma \subset \varphi(H(E))$, is a {\it weak quasi-model} of the triple $(F, G, \Gamma)$. If in addition the derivative of $\varphi$ does not vanish on any fiber, the triple $(E, H, \varphi)$ is called a {\it quasi-model}. The triple $(F, G, \Gamma)$ is called {\it minimal} if for any quasi-model $(E, H, \varphi)$ we necessarily have $E = F$ and $\varphi$ is an isomorphism. Now, if $(E, H, \varphi)$ is a quasi-model then the fact that the fibers of $H$ over factor fields of $E$ are adjoint makes the isogeny $\varphi$ purely inseparable. It follows that the induced map $H(E) \to G(F)$ is injective, which enables us to identify $\Gamma$ with its pre-image in $H(E)$. Then the triple $(E, H, \Gamma)$ satisfies the same assumptions as $(F, G, \Gamma)$. A (weak) quasi-model $(E, H, \varphi)$ is said to be {\it minimal} if the triple $(E, H, \Gamma)$ is minimal in the above sense. Pink \cite[Theorem 3.6]{Pink-CS} proves that every triple $(F, G, \Gamma)$ has a minimal quasi-model $(E, H, \varphi)$; the subring $E$ in this model is unique, and $H$ and $\varphi$ are determined up to unique isomorphism.
\vskip1mm

Now, let $(E, H, \varphi)$ be a minimal model of $(F, G, \Gamma)$, and view $\Gamma$ as a subgroup of $H(E)$. Let $\widetilde{H}$ be the universal cover of $H$ (it is the direct product of the universal covers of the factors of $H$). Then the commutator morphism of $\widetilde{H}$ factors through a unique morphism $[\ , \,]^{\sim} \colon H \times H \to \widetilde{H}$. Let $\widetilde{\Gamma}$ be the subgroup of $\widetilde{H}(E)$ generated by $[\Gamma , \Gamma]^{\sim}$.

\vskip2mm

\noindent {\bf Theorem.} (\cite[Main Theorem 0.2]{Pink-CS}) {\it The closure of \,$\widetilde{\Gamma}$ in $\widetilde{H}(E)$ is open.}

\vskip3mm

We can now state the key proposition that leads to the existence of generic elements. In the rest of this paper $G$ will denote a connected absolutely simple adjoint group defined over a finitely generated field $K$ and $\Gamma\subset G(K)$ a Zariski-dense subgroup that contains an element of infinite order. For a discrete valuation $v$ of $K$ such that the completion $K_v$ is locally compact and $\Gamma$  has compact closure in $G(K_v)$, we let $(E_v, H_v, \varphi_v)$ denote a minimal quasi-model of $(K_v, G, \Gamma)$. Let $r$ be the number of conjugacy classes in the Weyl group of (a maximal torus of) $G$.
\begin{prop}\label{P:Exist1}
Assume that there exist a subfield $K' \subset K$ such that $K/K'$ is a purely inseparable extension and valuations $v_1, \ldots , v_r$ of $K$ satisfying the following properties:

\vskip2mm

\noindent {\rm (0)} \parbox[t]{15cm}{Each completion $K_{v_i}$ is locally compact and $\Gamma$ has compact closure in $G(K_{v_i})$;}

\vskip1mm

\noindent {\rm (1)} \parbox[t]{15cm}{For each $i \in \{1, \ldots , r\}$, the group $H_{v_i}$ is $E_{v_i}$-split and $E_{v_i}$ contains $K'$;}

\vskip1mm

\noindent {\rm (2)} \parbox[t]{15cm}{Set $V = \{v_1, \ldots , v_r\}$, $K_V = \prod_{v \in V} K_v$, $E_V = \prod_{v \in V} E_v\,(\subset K_V)$, $H_{V} = \prod_{v \in V} H_v$, and $\varphi_{V} = \prod_{v \in V} \varphi_v$; then $(E_V, H_{V}, \varphi_{V})$ is a minimal quasi-model of $(K_V, G, \Gamma)$ $($here $\Gamma$ is diagonally embedded into $G(K_V) = \prod_{v \in V} G(K_v))$.}

\vskip2mm

\noindent Then $\Gamma'\,(\subset\Gamma)$ contains regular semi-simple elements of infinite order that are generic over $K$.
\end{prop}
\begin{proof}
For $v \in V$, we let $\varpi_v \colon \widetilde{H}_v \to H_v$ denote the universal $E_v$-cover, and set $\pi_v = \varphi_v \circ (\varpi_v)_{K_v}$. Since $H_v$ is $E_v$-split, $\widetilde{H}_v$ is also $E_v$-split. As above, we can unambiguously identify $\Gamma$ with a subgroup of $H_V(E_V) = \prod_{v \in V} H_v(E_v)$. Furthermore, for each $v \in V$, let  $[\ , \,]^{\sim}_v \colon H_v\times H_v \to \widetilde{H}_v$ be the $E_v$-morphism obtained from the commutator map, and let
$$
[\ , \,]^{\sim} = \prod_{v \in V} [\ , \,]^{\sim}_v
$$
be the product of these morphisms regarded either as a morphism of $E_V$-schemes $H_{V} \times H_{V} \to \widetilde{H}_V := \prod_{v \in V} \widetilde{H}_v$, or simply as a map
$$
H_{V}(E_V) \times H_{V}(E_V) \to \widetilde{H}_{V}(E_{V}) = \prod_{v \in V} \widetilde{H}_v(E_v).
$$
Let $\widetilde{\Gamma}$ be the subgroup of $\widetilde{H}_{V}(E_V)$ generated by $[\Gamma , \Gamma]^{\sim}$. Since by our assumption $(E_V, H_{V}, \varphi_{V})$ is a minimal quasi-model of $(K_V, G, \Gamma)$, the approximation theorem of Pink stated above tells us that the closure of $\widetilde{\Gamma}$ in $\widetilde{H}_{V}(E_V)$ is open.

\vskip1mm

Now, for $i\leqslant r$, let $\widetilde{T}_i$ be a maximal $E_{v_i}$-split torus of $\widetilde{H}_{v_i}$, and let $S_i = \pi_{v_i}((\widetilde{T}_i)_{K_{v_i}})$ be the corresponding maximal $K_{v_i}$-torus of $G$. We extend the associated comorphism
$$
\pi_i^* \colon {\rm{X}}(S_i) \to {\rm{X}}(\widetilde{T}_i)
$$
of the character groups to an isomorphism of vector spaces
$$
\tau_i: V_i := {\rm{X}}(S_i) \otimes_{\Z} \Q {\longrightarrow} {\rm{X}}(\widetilde{T}_i) \otimes_{\Z} \Q =: \widetilde{V}_i.
$$
We consider the automorphism groups of the root systems $\Phi(\widetilde{H}_{v_i} , \widetilde{T}_i)$ and $\Phi(G , S_i)$ as subgroups of $\mathrm{GL}(\widetilde{V}_i)$ and $\mathrm{GL}(V_i)$, respectively. Then by \cite[Prop.\,4]{Cheval}, the isomorphism $$\lambda_i \colon \mathrm{GL}(\widetilde{V}_i) \to \mathrm{GL}(V_i), \ \ \ g \mapsto \tau_i^{-1} \circ g \circ \tau_i$$ induces an isomorphism $W(\widetilde{H}_{v_i}, \widetilde{T}_i) \to W(G , S_i)$ of the Weyl groups.
\vskip1mm

We fix a maximal $K$-torus $S$ of $G$ and let $[w_1],\ldots, \,[w_r]$ be the distinct nontrivial conjugacy classes in the Weyl group $W(G,S)$.   We use $\iota_{S_i,S}$ to identify conjugacy classes in the Weyl group $W(G,S_i)$ with conjugacy classes in $W(G,S)$. For $i\leqslant r$, we pick ${\widetilde{w}}_i \in W(\widetilde{H}_{v_i} , \widetilde{T}_{i})$ so that $[\lambda_{i} ({\widetilde{w}}_i)]= [w_i]$.
Since $\widetilde{H}_{v_i}$ is $E_{v_i}$-split for all $i$, the argument used in \cite{PR-generic} to prove Lemma 1 (this argument  works in all characteristics) shows that  for $i \leqslant r$, one can find a maximal
$E_{v_i}$-torus $\widetilde{\mathcal{T}}_{i}$ of $\widetilde{H}_{v_i}$ such that
\begin{equation}\label{E:YYY1}
\theta_{\widetilde{\mathcal{T}}_{i}}(\Ga(E_{v_i}^{\mathrm{sep}}/E_{v_i})) \cap \iota_{\widetilde{T}_{i}, \widetilde{\mathcal{T}}_{i}}([{\widetilde{w}}_i]) \neq \varnothing.
\end{equation}
Let $\mathcal{S}_{i} =\pi_{v_i}(({\widetilde{\mathcal{T}}}_{i})_{K_{v_i}})$. Then $\mathcal{S}_{i}$ is a maximal $K_{v_i}$-torus of $G$. Let
$${\widetilde{\mathcal{U}}}_i ={\widetilde{\psi}}_i ({\widetilde{H}}_{v_i}(E_{v_i}) \times ({\widetilde{\mathcal{T}}}_{i})_{\mathrm{reg}}(E_{v_i})), \ \ \ {\mathrm{where}} \ \  {\widetilde\psi}_i: {\widetilde{H}}_{v_i}\times {\widetilde{\mathcal{T}}}_{i}\rightarrow {\widetilde{H}}_{v_i}, \ \ \ (\tilde{h},\tilde{t})\mapsto \tilde{h}\,\tilde{t}\,{\tilde{h}}^{-1},
$$ and $${{\mathcal{U}}}_i ={{\psi}}_i ({{G}}(K_{v_i}) \times ({{\mathcal{S}}}_{i})_{\mathrm{reg}}(K_{v_i})), \ \ \ {\mathrm{where}} \ \  {\psi}_i: {{G}}\times {{\mathcal{S}}}_{i}\rightarrow G, \ \ \ ({g},{s})\mapsto {g}{s}{{g}}^{-1}.
$$

Observe that by the Open Mapping Theorem, ${\widetilde{\mathcal{U}}}_i$ and $\mathcal{U}_i$ are open in $\widetilde{H}_{v_i}(E_{v_i})$ and $G(K_{v_i})$ respectively and they clearly intersect every open subgroup of the respective ambient groups. Let $\Omega$ be a compact-open subgroup of $G(K_V):=\prod _{v\in V}G(K_v)$ that does not contain any element whose $v$-component, for some $v\in V$,  is of finite order but not unipotent. Let $\widetilde\Omega$ be a compact-open  subgroup of $\widetilde{H}_V(E_V)$ that is contained in the inverse image of $\Omega$ under the continuous homomorphism ${\widetilde{H}}_V(K_V)\rightarrow G(K_V)$ induced by $\pi_V:= \prod_{v\in V}\pi_v$.   Since the closure of $\widetilde{\Gamma}$ is an open subgroup of $\widetilde{H}_V(E_V)$, we see that $\widetilde{\Gamma} \cap({\widetilde\Omega} \cap \prod_{i = 1}^r {\widetilde{\mathcal{U}}}_i) \neq \varnothing$. Let $\widetilde{\gamma}$ be an element of this intersection, and let $\gamma\,(\in \Omega \cap \prod_{i=1}^r \mathcal{U}_i)$ be the corresponding element of $\Gamma$. We note that as the subsemigroup $\Gamma'$ generates $\Gamma$, and the closure of the latter in $G(K_V)$ is a compact subgroup, the closure of $\Gamma'$ in $G(K_V)$ is a subgroup and so it contains $\Gamma$.  Now since $\Omega \cap \prod_{i=1}^r \mathcal{U}_i$ is an open neighborhood of $\gamma\,(\in \Gamma)$ in $G(K_V)$, $\Gamma'\cap(\Omega \cap \prod_{i=1}^r \mathcal{U}_i) \neq\varnothing$. Let $\gamma' = (\gamma'_1,\ldots , \gamma'_r)$, with $\gamma'_i\in \mathcal{U}_i$,  be an element of this intersection. This element is clearly of infinite order; we will now show that it is generic.    Let $\mathscr{S} = Z_G(\gamma')^{\circ}$; this is a maximal $K$-torus of $G$. Let ${\mathscr{S}}_{i} = Z_G(\gamma'_i)^{\circ}$. Then $\mathscr{S}_{i}$ is conjugate to $\mathcal{S}_{i}$ by an element of $G(K_{v_i})$, and moreover, ${\mathscr{S}}_{K_{v_i}} = \mathscr{S}_{i}$.
\vskip1mm

Since $K' \subset E_{v_i}$, $K/K'$ is purely inseparable, $\mathscr{S}_{K_{v_i}}= \mathscr{S}_i$ is conjugate to $\mathcal{S}_i$ by an element of $G(K_{v_i})$ and $\pi_{v_i}(({\widetilde{\mathcal{T}}}_i)_{K_{v_i}})={\mathcal{S}}_i$, it follows from (\ref{E:YYY1})  by applying $\pi_{v_i}$ that
\begin{equation}\label{E:YYY2}
\theta_{{\mathscr{S}}_{K_{v_i}}}(\Ga(K_{v_i}^{\mathrm{sep}}/K_{v_i})) \cap \iota_{S_{i},\, {\mathscr{S}}_{K_{v_i}}}([{{w}}_i]) \neq \varnothing.
\end{equation}
Thus, the image $\theta_{\mathscr{S}}(\Ga(K^{\mathrm{sep}}/K))\,( \subset \mathrm{Aut}\: \Phi(G , \mathscr{S}))$ intersects every conjugacy class of $W(G ,{\mathscr{S}})$, and therefore it contains $W(G , \mathscr{S})$.  So, $\gamma'$ is generic, as required.
\end{proof}

\vskip1mm

In applying the preceding proposition, condition (0) is easy to achieve while conditions (1) and (2) require more work. The subtlety of condition (1) is that while it is easy to construct valuations $v$ such that $G$ is split over $K_{v}$, this may not imply automatically that $H_v$ is $E_v$-split. More precisely, given a $K$-isogeny $\pi \colon H \to G$ of connected absolutely almost simple algebraic groups over a field $K$ of positive characteristic, $H$ need not be $K$-split when $G$ is unless $\pi$ is a central isogeny.\footnote{To construct an example, let $q$ be a ``non-degenerate'' quadratic form of defect 1 on a $(2n+1)$-dimensional vector space V  over a field $k$ of characteristic $2$. Then the induced bilinear form on  $V/\mathrm{Rad}(q)$ is a non-degenerate alternating form in $2n$ variables which is invariant under $\mathrm{SO}(q)$. Thus we get the isogeny $\mathrm{SO}(q)\to \mathrm{Sp}(2n)$. Now, over a locally compact field $k$, the form  $q$ can be chosen to be of Witt index $n-1$, so $\mathrm{SO}(q)$ is not $k$-split, but $\mathrm{Sp}(2n)$ is $k$-split.

Note that if the absolute root system of $G$ is simply-laced then $\pi$ is a central isogeny.} We note that over nondiscrete locally compact fields all groups of type $\textsf{F}_4$ and $\textsf{G}_2$ are split, so in our situation this problem can arise only for isogenies between groups of types $\textsf{B}_n$ and $\textsf{C}_n$ over fields of characteristic two. However treating just this case does not appear to be simpler then treating the general case, which is what we are going to do. We begin with two simple lemmas.
\begin{lemma}\label{L:Vin}
{\rm (cf.\,Vinberg \cite[Lemmas 2 and 3]{Vin})} Let $\Delta \subset M_n(K)$ be an absolutely irreducible multiplicative semi-group, and let $E$ be a subfield of $K$ such that $\mathrm{tr}\: \delta \in E$ for all $\delta \in \Delta$. Then the characteristic polynomial of every $\delta \in \Delta$ has coefficients in $E$.
\end{lemma}
(In characteristic zero this, of course, immediately follows from Newton's formulas.)
\begin{proof}
Let $A$ be the $E$-span of $\Delta$; clearly, $A$ is an $E$-algebra. We will first show that $A$ is an $E$-form of $M_n(K)$, i.e. $A \otimes_E K \simeq M_n(K)$. Indeed, since $\Delta$ is absolutely irreducible, by Burnside's theorem, we can pick $\delta_1, \ldots , \delta_{n^2} \in \Delta$ that are linearly independent over $K$. Set $B = \sum_{i = 1}^{n^2} E\delta_i$. Clearly, the map
$$
\tau \colon M_n(K) \to K^{n^2}, \ \ a \mapsto (\mathrm{tr}(a\delta_1), \ldots , \mathrm{tr}(a\delta_{n^2})),
$$
is an isomorphism of $K$-vector spaces. Since $E$ contains the traces of all elements of $\Delta$, we obtain that $\tau(\Delta) \subset E^{n^2}$ and the matrix of the trace form in the basis $\delta_1, \ldots , \delta_{n^2}$ has entries in $E$. It follows that $\Delta \subset B$, and therefore $A = B$; in particular, $\dim_E A = n^2$. Then the natural homomorphism $A \otimes_E K \to M_n(K)$ is clearly an isomorphism, implying that $A$ is a central simple $E$-algebra. So, the characteristic polynomial of $\delta \in \Delta \subset A$ can be viewed as its reduced polynomial, and therefore has coefficients in $E$.
\end{proof}

\vskip2mm

To formulate the next lemma, we need to introduce one additional technical notion. Let $\gamma \in G(K)$ be a regular semisimple element, and $T = Z_G(\gamma)^{\circ}$ be the corresponding maximal torus. We say that $\gamma$ is \emph{super-regular} if the values $a(\gamma)$, for $a \in \Phi(G , T)$, are all distinct. We note that the set of super-regular elements is Zariski-open.
\begin{lemma}\label{L:split}
Let $H$ be an absolutely almost simple algebraic group over a field $E\,(\subset K)$, and let $\varphi \colon H_K \to G \,(\subset \mathrm{GL}_n)$ be an isogeny. Let $\gamma \in H(E)$ be a semisimple element such that $\varphi(\gamma)$ is super-regular and has eigenvalues in $E$. Then $\gamma$ is regular and the corresponding torus $T = Z_H(\gamma)^{\circ}$ is $E$-split; in particular, $H$ splits over $E$.
\end{lemma}
\begin{proof}
Let $T$ be a maximal $E$-torus of $H$ containing $\gamma$, and let $S= \varphi(T_K)$. We let $\Phi = \Phi(G , S)$ and $\Phi' = \Phi(H , T)$ denote the corresponding root systems. Set $p = 1$ if $\mathrm{char}\: E =  0$, and $p = \mathrm{char}\: E$ otherwise. Chevalley \cite[p.\,5]{Cheval} proves that there exists a bijection $\psi \colon \Phi \to \Phi'$ such that
\begin{equation}\label{E:YYY03}
\varphi^*(a) = p^{d(a)} \psi(a) \ \ \text{for all} \ \  a\in \Phi,
\end{equation}
where $d(a)$ is an integer $\geqslant 0$. Since
$$
\varphi^*(a)(\gamma) = a(\varphi(\gamma))  \ \ \text{for any} \ \ a \in \Phi,
$$
and $\varphi(\gamma)$ is regular, it follows from (\ref{E:YYY03}) that, first, for all $b \in \Phi'$, $b(\gamma) \neq 1$, and hence $\gamma$ is regular.
Moreover, since $d(a)$ is the same integer for all roots $a$ of a given length (which follows from the fact that the Weyl group acts transitively on the roots of the same length), we see that the values $b(\gamma)$, for $b \in \Phi'_{\mathrm{short}}$, are all distinct (we set $\Phi'_{\mathrm{short}} = \Phi'$ if all roots have the same length). Second, $b(\gamma) \in E^{1/p^{\infty}}$. At the same time, $b(\gamma)$ lies in a separable closure $E^{\mathrm{sep}}$ of $E$, so in fact $b(\gamma) \in E$ for all $b \in \Phi'$. Then for any $\sigma \in \mathscr{G} := \Ga(E^{\mathrm{sep}}/E)$ we have
$$
(\sigma(b))(\gamma) = \sigma(b(\sigma^{-1}(\gamma))) = b(\gamma).
$$
It follows that $\sigma(b) = b$ for all $b \in \Phi'_{\mathrm{short}}$ and all $\sigma \in \mathscr{G}$. Since  $\Phi'_{\mathrm{short}}$ span ${\rm{X}}(T)$, we obtain that $\mathscr{G}$ acts on ${\rm{X}}(T)$ trivially, i.e., $T$ is $E$-split.
\end{proof}

\vskip2mm

We will use the above two lemmas in the proof of Theorem 2 to verify condition (1) in Proposition \ref{P:Exist1}. We will now address condition (2) in this proposition.
\begin{prop}\label{P:minimal}
Let $V = \{v_1, \ldots , v_r\}$ be a finite set of discrete valuations of $K$ with locally compact completions, and for each $v \in V$ let $(E_v, H_v, \varphi_v)$ be a minimal quasi-model of $(K_v, G, \Gamma)$. As above, let $K_V = \prod_{v \in V} K_v$, $E_V = \prod_{v \in V} E_v$, $H_V = \prod_{v \in V} H_v$, and $\varphi_V = \prod_{v \in V} \varphi_v$. If the fields $E_v$ are pairwise non-isomorphic as topological fields then
$(E_V, H_V, \varphi_V)$ is a minimal quasi-model of $(K_V, G, \Gamma)$.
\end{prop}
\begin{proof}
%
%
Let $(E, H, \varphi)$ be a quasi-model of  $(E_V, H_V, \Gamma)$, where $\Gamma$ is identified with its lift via $\varphi_V$. We need to show that $E = E_V$ and $\varphi$ is an isomorphism. Write $E = \prod_{i = 1}^d E_{v_i}$ and $H = \prod_{i = 1}^d H_i$, where $E_i$ is a local field and $H_i$ is a connected  absolutely simple adjoint $E_i$-group. It is enough to show that $d = r$. Indeed, then, by analyzing idempotents, we see that after a possible reindexing  of the $E_i$'s we may assume that $E_i \subset E_{v_i}$. In this case, for each $i \in \{1, \ldots , r\}$, the triple $(E_i, H_i, \varphi_i)$, where $\varphi_i$ is the restriction of $\varphi$, is a quasi-model of  $(E_{v_i}, H_{v_i}, \Gamma)$. So, the minimality of the latter implies that $E_i = E_{v_i}$ and $\varphi_i$ is an isomorphism, hence the required result.

Now, if $d < r$ then some $E_{i_0}$ has nontrivial projections to $E_{v_i}$ and $E_{v_j}$ for some $i , j \in \{1, \ldots , r \}$, $i \neq j$. So, $(E_{i_0}, H_{i_0}, \Gamma)$ is a model of both $(E_{v_i}, H_{v_i}, \Gamma)$ and $(E_{v_j}, H_{v_j}, \Gamma)$. Since these models are minimal, we have
$$
E_{v_i} = E_{i_0} = E_{v_j},
$$
contradicting our assumption.
\end{proof}

The final preparatory step for the proof of Theorem 2 provides a construction of valuations with the required properties.
\begin{lemma}\label{L:Embed}
Let $K$ be a finitely generated field, $F$ an infinite subfield of $K$, and $R \subset K$ be a finitely generated subring. Then there exists a subfield $K' \subset K$ containing $F$ such that the extension $K/K'$ is purely inseparable and for any $r \geqslant 1$ one can find $r$ discrete valuations $v_1, \ldots , v_r$ of $K$ such that

\vskip2mm

\noindent {\rm (1)} \parbox[t]{15cm}{for each $i = 1, \ldots , r$, the completion $K_{v_i}$ is locally compact, the ring $R$ is contained in the valuation ring $\mathcal{O}(K_{v_i})$, and the completions of $F$ and $K'$ with respect to the restrictions of $v_i$ (i.e. the closures of $F$ and $K'$ in $K_{v_i}$) coincide;}
\vskip1mm

\noindent {\rm (2)} \parbox[t]{15cm}{for $i \neq j$, the residue fields of $K_{v_i}$ and $K_{v_j}$ have different sizes.}
\end{lemma}
\begin{proof}
We only need to consider the case where $K$ has characteristic $p > 0$. Pick a separable transcendence basis $s_0, \ldots , s_a$ of $F$ over the prime subfield $\mathbb{F}_p$ ($a \geqslant 0$ since $F$ is infinite), and let $t_1, \ldots , t_b$ be any transcendence basis of $K/F$. We then let $K'$ denote the separable closure of $F(t_1, \ldots , t_b)$ in $K$. Then $K'$ is a finite separable extension of $L = \mathbb{F}_p(s_0, \ldots , s_a, t_1, \ldots , t_b)$, and $K/K'$ is a finite purely inseparable extension. Since $R$ is finitely generated, we can find a nonzero $h \in C := \mathbb{F}_p[s_0, \ldots , s_a, t_1, \ldots , t_b]$ such that all elements of $R$ are integral over $C_h: = C[1/h]$.
\vskip1mm

Let $\alpha \in K'$ be a primitive element over $L$. We may assume without loss of generality that the minimal polynomial of $\alpha$ is of the form
$$
f(x) = x^n + p_{n-1} x^{n-1} + \cdots + p_0 \ \ \text{with} \ \ p_i \in C.
$$
Set $A = \mathbb{F}_p[s_0]$ and $k = \mathbb{F}_p(s_0)$, and then think of the $p_i$'s as elements of $C = A[s_1, \ldots , s_a, t_1, \ldots , t_b]$. Let $q = q(s_1, \ldots , t_b) \in A[s_1, \ldots , t_b]$ be the discriminant of $f$; note that $q \neq 0$ as $f$ is separable. We then pick $s^0_1, \ldots , t^0_b \in A$ so that $q(s^0_1, \ldots , t^0_b) \neq 0$ and $h(s_0, s^0_1, \ldots , t^0_b) \neq 0$, and let
$$
f_0(x) = x^n + p_{n-1}(s^0_1, \ldots , t^0_b)x^{n-1} + \cdots + p_0(s^0_1, \ldots , t^0_b) \in A[x].
$$
By our construction, $f_0(x)$ is a separable polynomial. It follows from Chebotarev's Density Theorem that one can find discrete valuations $v^0_1, \ldots v^0_r$ of $k$ corresponding to the irreducible polynomials in $A$ of pairwise distinct degrees such that for each $j \in \{ 1, \ldots , r \}$ the residue polynomial
$\overline{f_0(x)}^{(v^0_j)}$ over the residue field $\kappa_{v^0_j}$ is separable and splits into linear factors, and the residue $\overline{h(s_0, s^0_1, \ldots , t^0_b)}^{(v^0_j)} \neq 0$ in $\kappa_{v^0_j}$.

Let us show that for each $j = 1, \ldots , r$, there exists an embedding $$\iota_j \colon K \hookrightarrow \overline{k_{v^0_j}} =: \mathcal{K}_j \ \ \text{(algebraic closure of} \:k_{v^0_j})$$
extending the standard embedding of $k$ such that $\iota_j(K') \subset k_{v^0_j}$ and $\iota_j(R)$ is contained in the valuation ring $\mathcal{O}(\mathcal{K}_j)$. We let $\mathfrak{p}_j$ denote the valuation ideal in $k_{v^0_j}$. Then one can find elements $\widetilde{s}_1, \ldots , \widetilde{s}_a, \widetilde{t}_1, \ldots , \widetilde{t}_b \in k_{v^0_j}$ that are {\it algebraically independent over} $k$ and congruent respectively to $s^0_1, \ldots , s^0_a, t^0_1, \ldots , t^0_b$ modulo $\mathfrak{p}_j$. This enables us to construct an embedding of $L = k(s_1, \ldots , s_a, t_1, \ldots , t_b)$ into $k_{v^0_j}$ sending $s_1, \ldots , s_a, t_1,\ldots , t_b$ to $\widetilde{s}_1, \ldots , \widetilde{s}_a, \widetilde{t}_1,\ldots,  \widetilde{t}_b$. Next, we observe that  the polynomial
$$
\widetilde{f}(x) := x^n + p_{n-1}(\widetilde{s}_1, \ldots , \widetilde{s}_a, \widetilde{t}_1,\ldots , \widetilde{t}_b)x^{n-1} + \cdots + p_0(\widetilde{s}_1, \ldots ,\widetilde{s}_a, \widetilde{t}_1, \ldots , \widetilde{t}_b)
$$
has a root in $k_{v^0_j}$. Indeed, the residue $\overline{\widetilde{f}(x)}^{(v^0_j)}$ coincides with $\overline{f^0(x)}^{(v^0_j)}$, hence is a product of distinct linear factors over $\kappa_{v^0_j}$. So, the fact that $\widetilde{f}(x)$ has a root in $k_{v^0_j}$ follows from Hensel's Lemma, and in turn implies that the above embedding $L \hookrightarrow k_{v^0_j}$ extends to an embedding $K' \hookrightarrow k_{v^0_j}$. Now, for the required embedding $\iota_j$ we take the unique extension of the latter to $K$. We only need to show that $\iota_j(R) \subset \mathcal{O}(\mathcal{K}_j)$. According to our construction, we have the inclusion $\iota_j(C) \subset \mathcal{O}(k_{v^0_j}) \subset \mathcal{O}(\mathcal{K}_j)$. Furthermore, $\iota_j(h) = h(s_0, \widetilde{s}_2, \ldots , \widetilde{t}_b)$ is a unit in $\mathcal{O}(k_{v^0_j})$, so $\iota_j(C_h) \subset \mathcal{O}(\mathcal{K}_j)$. Since every element of $R$ is integral over $C_h$, the inclusion $\iota_j(R) \subset \mathcal{O}(\mathcal{K}_j)$ follows.
\vskip1mm

Now, let $v_j$ is the pullback to $K$ (via $\iota_j$) of the standard valuation on $\mathcal{K}_j$. Since $\iota_j(K') \subset k_{v^0_j}$ and $K/K'$ is finite, the completion $K_{v_j}$ is locally compact. All other properties in (1) immediately follow from our construction. Furthermore, by our construction, for $j \neq j$, the local fields $k_{v^0_i}$ and $k_{v^0_j}$ have the residue fields of different sizes. Since $K_{v_i}$ and $K_{v_j}$ are purely inseparable extensions of these fields while the residue fields are perfect, (2) follows.
\end{proof}

{\it Proof of Theorem 2.} Let $\pi \colon G \to \overline{G}$ be a central $K$-isogeny onto the corresponding adjoint group. It is easy to see that if $\gamma' \in \Gamma'$ is such that $\pi(\gamma')$ is a regular semi-simple element of infinite order that is generic over $K$ then $\gamma'$ possesses all these properties as well. Thus, we may assume from the beginning that $G$ is adjoint. Next, as we have seen at the beginning of this section, we may assume that $\Gamma'$ is finitely generated. Fixing a faithful $K$-representation $G \hookrightarrow \mathrm{GL}_n$, we can find a finitely generated ring $R$ in $K$ so that $\Gamma' \subset \mathrm{GL}_n(R)$. Let $\mathfrak{g} = L(G)$ be the Lie algebra of $G$. Any nontrivial $K$-isogeny $\varphi \colon H \to G$, with $H$ connected and adjoint, is purely inseparable and the image of the differential $d\varphi$ is either zero or contains the unique irreducible ${\mathrm{Ad}}\:G$-submodule $\mathfrak{m}$ of $\mathfrak{g}$. Let $\rho \colon G \to \mathrm{GL}(\mathfrak{m})$ denote the corresponding representation. Let $F$ be the subfield of $K$ generated by the traces $\mathrm{tr}\: \rho(\gamma)$, $\gamma\in\Gamma$; clearly, $F$ is infinite. Pick a super-regular $\gamma_0 \in \Gamma$; then  $\rho(\gamma_0)$  is super-regular in $\rho(G)$. Let $\chi(t)$ be the characteristic polynomial of $\rho(\gamma_0)$ which by Lemma \ref{L:Vin} has coefficients in $F$. Write $\chi(t) = (t - 1)^a f(t)$ where $f(t) \in F[t]$ is such that $f(1) \neq 0$. Since $\gamma_0$ is super-regular in $G$, the polynomial $f(t)$ does not have multiple roots. Expanding $K$, we may assume that $f$ splits over $K$ into linear factors. Then, since $f$ is separable, for any subfield $K' \subset K$ containing $F$ and such that $K/K'$ is purely inseparable, the polynomial $f$ splits into linear factors already over $K'$.

Now, using Lemma \ref{L:Embed}, we find a subfield $K' \subset K$ containing $F$ such that $K/K'$ is purely inseparable and discrete valuations $v_1, \ldots , v_r$ (where $r$ is the number of nontrivial conjugacy classes in the Weyl group of $G$) of $K$ satisfying conditions (1) and (2) therein. Set $V = \{v_1, \ldots , v_r\}$. Then for any $v \in V$, the completion $K_v$ is locally compact by construction and the closure of $\Gamma$ in $G(K_v)$ is compact due to the inclusions $\Gamma \subset \mathrm{GL}_n(R)$ and $R \subset \mathcal{O}(K_v)$, verifying condition (0) of Proposition \ref{P:Exist1}. Let $(H_v, E_v, \varphi_v)$ be a minimal quasi-model of $(G, K_v, \Gamma)$. Since the representation $\rho \circ \varphi_v$ is contained in the adjoint representation of $H_v$, we obtain from Proposition 3.10 of  \cite{Pink-CS} that $E_v \,(\subset K_v)$ contains $F$. Since $F$ and $K'$ have the same closure in $K_v$ and $f$ splits over $K'$ into linear factors, we conclude that all eigenvalues of $\rho(\gamma_0)$ lie in $E_v$. On the other hand, by the definition of a quasi-model, there exists $\gamma \in H_v(E_v)$ such that $\varphi_v(\gamma) = \gamma_0$. Applying Lemma \ref{L:split} to the isogeny $\rho \circ \varphi_v \colon H_v \to \rho(G)$, we obtain that $H_v$ is $E_v$-split, which verifies condition (1) of Proposition \ref{P:Exist1}. Finally, as we have seen, $E_v$ contains $K'$, and therefore the extension $K_v/E_v$ is purely inseparable. So, since the fields $K_{v_j}$ for $j= 1, \ldots , r$ have finite residue fields of pairwise different sizes, the same is true for the fields $E_{v_j}$, making these fields pairwise non-isomorphic. Applying Proposition \ref{P:minimal}, we see that $(E_V, H_V, \varphi_V)$ is a minimal model of  $(K_V, G, \Gamma)$, verifying condition (2) of Proposition \ref{P:Exist1}. Now, the assertion of Theorem 2 on the existence of generic elements immediately follows from Proposition \ref{P:Exist1}. \hfill $\Box$

\bigskip

\noindent {\small {\bf Acknowledgements.} Both authors were supported by NSF through grants DMS-1401380 and DMS-1301800. The second-named author was also supported by the Humboldt and the Simons Foundations. Part of the paper was written in the summer of 2016 when he visited the University of Bielefeld whose hospitality is thankfully acknowledged.}

\bigskip

\bibliographystyle{amsplain}

\end{document}